\documentclass[11pt,a4paper]{article}
\usepackage{amsmath,amsthm,amssymb,amscd,epsfig,calc,mathrsfs}
\usepackage[all]{xy}
\usepackage{layout}
\normalsize
\theoremstyle{definition}
\numberwithin{equation}{section}
\newtheorem{thm}{Theorem}[section]
\newtheorem{dfn}[thm]{Definition}

\newtheorem{prop}[thm]{Proposition}

\newtheorem{lem}[thm]{Lemma}

\newtheorem{rem}[thm]{Remark}
\newcommand{\exam}{\noindent{\textbf{Example}:}}
\newcommand{\rmk}{\noindent{\textbf{Remark}:}}
\newcommand{\obs}{\noindent{\textbf{Observation}:}}
\newcommand{\ntt}{\noindent{\textbf{Notation}:}}
\newcommand{\ack}{\noindent{\textbf{Acknowledgement}:}}

\newcommand{\mf}[1]{{\mathfrak{#1}}}
\newcommand{\msf}[1]{{\mathsf{#1}}}

\newcommand{\bb}[1]{{\mathbb{#1}}}
\newcommand{\mca}[1]{{\mathcal{#1}}}
\newcommand{\bracket}[1]{{\langle {#1} \rangle}}

\def\Aut{\mathop{\mathrm{Aut}}\nolimits}
\def\AutH{\mathop{\mathrm{Aut}_{\mathrm{Hodge}}}\nolimits}

\def\deg{\mathop{\mathrm{deg}}\nolimits}
\def\det{\mathop{\mathrm{det}}\nolimits}
\def\dim{\mathop{\mathrm{dim}}\nolimits}
\def\div{\mathop{\mathrm{div}}\nolimits}
\def\Div{\mathop{\mathrm{Div}}\nolimits}

\def\Fix{\mathop{\mathrm{Fix}}\nolimits}

\def\id{\mathop{\mathrm{id}}\nolimits}

\def\Pic{\mathop{\mathrm{Pic}}\nolimits}

\def\rank{\mathop{\mathrm{rank}}\nolimits}

\def\Sing{\mathop{\mathrm{Sing}}\nolimits}

\title{Enriques surfaces covered by Jacobian Kummer surfaces\thanks{2000 
Mathematics Subject Classification, 14J28}}
\author{Hisanori Ohashi\thanks{Kyoto University, e-mail: pioggia@kurims.kyoto-u.ac.jp}}
\date{June 19, 2008}

\begin{document}
\maketitle
\begin{abstract}
This paper classifies Enriques surfaces whose $K3$-cover is a fixed Picard-general 
Jacobian Kummer surface. There are exactly $31$ such surfaces.
We describe the free involutions which give these Enriques surfaces explicitly.
As a biproduct, we show that $\Aut (X)$ is generated by elements of order $2$,
which is an improvement of the theorem of S. Kondo.
\end{abstract}

\section{Introduction}

A $K3$ surface is a simply connected compact complex surface whose 
canonical bundle is trivial. Every Enriques surface appears as a 
quotient of a $K3$ surface by a fixed-point-free (shortly, free) involution.
Theoretically, to consider an Enriques surface is equivalent to
consider the pair of the covering $K3$ surface and the free involution.
For example, the period 
map for Enriques surfaces is constructed under this description.
But the properties of free involutions on a {\em fixed} $K3$ surface are 
rather unclear to us. The existence is already a special property, 
their geometric realizations and the isomorphism classes of 
the quotient Enriques surfaces are other problems. 

For a fixed $K3$ surface $X$, two quotient Enriques surfaces are isomorphic if and only if 
the two free involutions are conjugate in $\Aut (X)$. In \cite{ohashi} it is shown 
that the number of the conjugacy classes of 
free involutions (and more generally, of finite subgroups) are finite.
There this number, i.e., the number of isomorphism classes of quotient Enriques surfaces,
is computed for $K3$ surfaces with Picard number $\rho=11$ or for
Kummer surfaces associated with the product of two elliptic curves
whose periods are very general.

The aim of this paper is to study fixed-point-free
involutions on surfaces studied in \cite{keum97,kondo98}.
Let $C$ be a smooth projective curve of genus $2$.
Its Jacobian variety $J(C)$ is the abelian surface parametrizing
divisor classes on $C$ of degree $0$.
The quotient surface $J(C)/\{\pm 1_{J(C)}\}$ has $16$ nodes
and can be embedded into $\bb{P}^3$ as a quartic hypersurface.
We call it the {\em{Kummer quartic surface}} associated with $C$ and 
denote by $\overline{Km} (J(C))=:\overline{X}$.
The minimal desingularization 
$Km(J(C))=:X$ of $\overline{Km} (J(C))$ is called the {\em{Jacobian Kummer surface}}
associated with $C$, which is a $K3$ surface. 
$X$ is {\em{Picard-general}} if the Picard number of $X$ equals $17$,
the minimum possible value. In what follows, $X$ will always be a Picard-general 
Jacobian Kummer surface except for Sections \ref{method} and \ref{166}.

In \cite{mukai}, Mukai observed that there exist three kinds of free involutions on $X$.
\begin{itemize}
\item {\em{A switch}} associated with {\em{an even theta characteristic $\beta$}}.
\item {\em{A Hutchinson-G\"{o}pel (shortly HG) involution}} associated 
with {\em{a G\"{o}pel tetrad $G$}}.
\item {\em{A Hutchinson-Weber (shortly HW) involution}} associated with {\em{a Weber hexad $W$}}.
\end{itemize}

Essentially these automorphisms date back more than a century, but 
their freeness are found only recently in comparison.
Mukai studied HG involutions
in connection with the numerically reflective involutions of Enriques surfaces.
Also he conjectured that these are the all free involutions on $X$.
In this paper we prove the following theorem and confirm the conjecture.
\begin{thm}\label{thm1}
On a Picard-general Jacobian Kummer surface $X$, 
there are exactly $31=10+15+6$ free involutions up to conjugacy in $\Aut (X)$.
$10$ are switches, $15$ are HG involutions and $6$ are HW involutions.
\end{thm}

In \cite{kondo98}, Kondo proved that $\Aut (X)$ is generated by 
$32$ translations and switches,
$32$ projections and correlations,
$60$ HG involutions, 
and $192$ Keum's automorphisms.
One point of the proof was that $192$ Keum's automorphisms did not 
correspond in one-to-one way to 
the $192$ facets of the polyhedral cone introduced by
Borcherds and Kondo. Moreover they had infinite order while the others had order $2$.
In this respect, it can be expected that there exist $192$ {\em involutions} which 
correspond in one-to-one way to the $192$ facets of the polyhedral cone and
together with
the $32+32+60$ involutions they generate $\Aut (X)$. 
In fact, the HW involutions work well.
\begin{thm}\label{thm2}
$\Aut (X)$ is generated by the following {\em involutions}:
translations, switches, projections, correlations, HG involutions and HW involutions.
\end{thm}
This is a biproduct of the proof of Theorem \ref{thm1}.

The proof of Theorem \ref{thm1} is given in the following way. 
In Section \ref{method}
we introduce an invariant of a free involution, called {\em a patching subgroup},
which is a subgroup of $A_{NS(X)}=NS(X)^*/NS(X)$.
This subgroup appears naturally in the light of Nikulin's theory of lattices \cite{nikulin-sym}.
Under some condition,
we can show the invariance of the patching subgroup under conjugations.
Section \ref{invariants}, Proposition \ref{fin} shows conversely two free involutions
are conjugate if their patching subgroups are the same,
when $X$ is a Picard-general Jacobian Kummer surface. 
Simultaneously we see that $X$ has no more than $31$ non-isomorphic Enriques quotients.
These two Sections
reduce the proof of Theorem \ref{thm1} to concrete computations of patching subgroups 
of free involutions itemized above. 
The occurence of $31$ distinct patching subgroups shows Theorem \ref{thm1}.
The computations are worked out in 
Sections \ref{switch}-\ref{reye}.
The result shows that the generators of patching 
subgroups are expressed in terms of the classical notions.
It is summarized as follows.

In the switch case, let $\beta$ be an even theta characteristic
and $\sigma_{\beta}$ be the switch. $\beta$ corresponds to 
a pair of Rosenhain subgroups $R_1,R_2$. Then the patching 
subgroup $\Gamma_{\sigma_{\beta}}$ is cyclic of order $4$ and generated by
\[H/4 + \sum_{\alpha\in R_1} N_{\alpha}/2. \]
Of course we obtain the same group after replacing $R_1$ by $R_2$ in this case.

In the HG involution case, let $G$ be a G\"{o}pel tetrad and 
$\sigma_G$ be the HG involution. Then the patching subgroup $\Gamma_{\sigma_{G}}$
is $2$-elementary abelian of order $4$ and generated by 
\[H/2\text{ and }\sum_{\alpha\in G} N_{\alpha}/2. \]
We remark that this result of HG involution case also follows from the computations 
of \cite{mukai}.

In the HW involution case, let $W$ be a Weber hexad and 
$\sigma_W$ be the HW involution. Then the patching subgroup $\Gamma_{\sigma_{W}}$
is cyclic of order $4$ and generated by 
\[H/4 + \sum_{\alpha\in W} N_{\alpha}/2. \]
The divisors $H, N_{\alpha}\in NS(X)$ and 
also the classical notions appeared here will be defined in Section \ref{166},
where we recall the basic properties of Jacobian Kummer surfaces.
After fixing the basis of $A_{NS(X)}$, we can easily check that 
there appear $31$ distinct patching subgroups.

\medskip
{\ntt} We refer the readers to \cite{nikulin-sym} for the basic 
properties of the finite quadratic form $(A_L,b_L,q_L)$ associated with an even 
nondegenerate lattice $L$. By definition, $A_L$ is the finite abelian group $L^*/L$, 
$b_L: A_L\times A_L \rightarrow \bb{Q}/\bb{Z}$ is the symmetric bilinear form
and $q_L: A_L\rightarrow \bb{Q}/2\bb{Z}$ is the quadratic form, both naturally induced from 
that of $L$.
Usually we denote finite forms by $(A_L,q_L)$, omitting $b_L$, or only by $A_L$.

The hyperbolic plane is denoted by $U$, the root lattices $A_l,D_m,E_n$
are considered to be negative definite.
The rank one lattice $\bracket{2n}$ is also used in this paper.
On finite forms, $u(2)$ is the associated form of the lattice $U(2)$,
$\bracket{1/2n}$ is that of $\bracket{2n}$.
The set of generators $\{e,f\}$ of $u(2)$ satisfying 
\[q(e)=q(f)=0, b(e,f)=1/2\]
is called the standard generator.

For a lattice $T$ and $k=\bb{Q},\bb{C}$ we denote the scalar extension 
by $T_k$.
If $T$ is a lattice and $T_{\bb{C}}$ is equipped with a Hodge 
structure, then $\AutH (T)$ is 
a subgroup of $O(T)$ whose elements preserve the Hodge decomposition.

\section{The method of counting}\label{method}

In this section $X$ is any $K3$ surface.
Let $\sigma$ be a free involution on $X$.
The $(-1)$-eigenspace of the action of $\sigma$ on $NS(X)$ is denoted by $K$.
Then it is well-known that $K$ is negative definite, contains no $(-2)$-element 
and the primitive hull of $K\oplus T_X$ in $H^2(X,\bb{Z})$ is isometric to 
$U\oplus U(2)\oplus E_8(2)=:N$. We choose a marking $\phi: \overline{K\oplus T_X}\rightarrow N$
for this isometry.

The nonzero global holomorphic $2$-form $\omega_X$ on $X$ determines via $\phi$ a point 
in $\mca{D}(N)/O(N)$, which is the period of the Enriques surface $Y:=X/\sigma$, 
where 
\[ \mca{D}(N):=\{\bb{C}\omega \in \bb{P}(N_{\bb{C}})|
   \omega\in N\otimes \bb{C},\omega\cdot \omega=0,\omega\cdot\overline{\omega}>0\}\]
is the (two copies of)
bounded symmetric domain of type IV associated to lattice $N$ of 
signature $(2,10)$. Obviously this period is independent of the choice of $\phi$ and 
the Torelli theorem of Enriques surfaces says that 
this point determines the isomorphism class of $Y$ uniquely.

Conversely given a primitive embedding $\phi: T_X\rightarrow N$ such that
the orthogonal complement $K$ is free from $(-2)$-elements, 
by the surjectivity there exists an Enriques surface $Y$ 
whose period is exactly $[\phi(\bb{C}\omega_X)]$.
If $\rho (X)\ge 12$ then  
\cite[Theorem 1]{keum90} shows that $X$ is isomorphic to the universal double
cover of $Y$. Even if $\rho (X)\le 11$ the same holds,
whose proof is in \cite{ohashi07pre}. 

Thus we have shown
\begin{prop}\label{correspondence}
There is a one-to-one correspondence between the sets
\[\{{\mathrm{Enriques\ quotients\ of\ }}X\}/{\mathrm{(isomorphisms)}}\]
and 
\[\left\{
\begin{array}{l}
{\mathrm{Primitive\ embeddings\ }}\phi :T_X\rightarrow N\\
{\mathrm{such\ that\ }}K=T_X^{\perp}{\mathrm{contains\ no\ }}(-2){\mathrm{-elements}}
\end{array}
\right\}
\big/{\mathrm{(Hodge\ isometries\ of\ }}N{\mathrm{)}},\]
where for each $\phi$ we equip $N$ with a Hodge structure induced from that of $T_X$ 
by $\phi$.
\end{prop}

In the following, we identify $\overline{K\oplus T_X}$ with $N$ by $\phi$.
By \cite{nikulin-sym}, there are subgroups $\Gamma_K\subset A_K$ and $\Gamma_{T_X}\subset 
A_{T_X}$ and a sign-reversing isometry $\varphi: \Gamma_K \stackrel{\sim}{\rightarrow}
\Gamma_{T_X}$ such that $N$ is the sublattice of $K_{\bb{Q}}\oplus T_{X,\bb{Q}}$
generated by $K$, $T_X$ and $\{(x,\varphi(x))|x\in \Gamma_K\}$.
\begin{dfn}\label{df}
The {\em{patching subgroup}} $\Gamma_{\sigma}$ of the free involution $\sigma$
is the inverse image of $\Gamma_{T_X}$ by the natural sign-reversing isometry 
$A_{NS(X)}\stackrel{\sim}{\rightarrow} A_{T_X}$.
\end{dfn}

Under a condition, $\Gamma_{\sigma}$ is an invariant of a conjugacy class
which is very computable.
\begin{prop}\label{invariance}
If $\AutH (T_X)=\{\pm \id\}$, then $\Gamma_{\sigma}$ depends only on the 
isomorphism class of the quotient Enriques surface.
\end{prop}
\begin{proof}
By Proposition \ref{correspondence}, conjugate free involutions induce on $N$
an isometric Hodge structure.
Any Hodge isometry of $N$ preserves $K=\omega_X^{\perp}$ and
hence $T_X$. Thus it induces $\pm \id$ on $T_X$ and preserves the subgroup $\Gamma_{T_X}$.
\end{proof}

\begin{rem}
The condition above is weak. It is true if $\rho (X)$ is odd, see \cite[p597]{kondo98},
or even if $\rho (X)$ is even, it is true if 
$X$ is very general in the period domain (\cite[Proposition 3.1]{ohashi}).
\end{rem}

In general there are free involutions not conjugate each other but with 
the same $\Gamma_{\sigma}$. However in the Picard-general Jacobian Kummer case, 
$\Gamma_{\sigma}$ completely classifies free involutions. 
This will be shown in the next section.

The computation of $\Gamma_{\sigma}$ is done by
\begin{lem}\label{gamma}
Let $\sigma ,K$ as above. Then 
\[\Gamma_{\sigma}=\{[x]\in NS(X)^*/NS(X)|\exists [y]\in K^*/K,\ x-y\in NS(X)\}.\]
\end{lem}
\begin{proof}
Let $\rho : A_{NS(X)}\rightarrow A_{T_X}$ be the canonical isomorphism.
Then $\rho ([x])=[z]$ is equivalent to $x+z\in H^2(X,\bb{Z})$.
Since 
$\Gamma_{T_X}=\{[z]\in T_X^*/T_X|\exists [y]\in K^*/K,\ y+z\in N\}$,
\begin{eqnarray*}
\Gamma_{\sigma} &=& \{ [x]\in NS(X)^*/NS(X)| \rho ([x])\in \Gamma_{T_X}\}\\
\ &=& \{[x]\in NS(X)^*/NS(X)|\exists [y]\in K^*/K,\ x-y\in NS(X)\}.
\end{eqnarray*}
This is what we need.
\end{proof}

\section{Invariants of free involutions}\label{invariants}

Let $C$ be a genus $2$ curve, $J(C)$ its Jacobian and $Km(J(C))=X$ 
the associated Jacobian Kummer surface as in the Introduction.
As is well-known,
$J(C)$ contains $C$ as a theta divisor:
\[\Theta=\{[p-p_0] | p\in C \}\subset J(C), \quad p_0\in C.\]
Hence $\rank NS(J(C))\ge 1$ and $\rank NS(X)\ge 17$
holds. When we have the equality, we call $X$ {\em{Picard-general}}.
In this case, since $T_{J(C)}=U^{\oplus 2}\oplus \bracket{-2}$ we have
$T_X=U(2)^{\oplus 2}\oplus \bracket{-4}$
and $NS(X)=U\oplus D_4^{\oplus 2}\oplus D_7$.

For simplicity, we put $T:=T_X$. 
Suppose we are given a primitive embedding of $T$ into 
$N$ such that the orthogonal complement is free from $(-2)$-elements, 
as in Proposition \ref{correspondence}.
First we determine the orthogonal complement.
\begin{prop}\label{e7}
The lattice $K=T_X^{\perp}$ is isometric to $E_7(2)$.
\end{prop}
\begin{proof}
Consider the unique embedding of $N$ into the abstract $K3$ lattice $L$. 
The orthogonal complement is denoted by $M$, $M\simeq U(2)\oplus E_8(2)$.
By \cite{nikulin-sym}, we have the following isomorphism of discriminant 
quadratic forms:
\begin{equation}
-q_K \simeq (q_M\oplus q_T|_{\Gamma^{\perp}})/\Gamma\label{nikulin}
\end{equation}
where $\Gamma$ is the pushout (i.e. the graph) of a sign-reversing isometry of subgroups 
$\Gamma_M\subset A_M$ and $\Gamma_T\subset A_T$.

For a finite quadratic form $(A,q)$, we denote the quadratic form induced on the 
$2$-torsion subgroup $A_2=\{x\in A| 2x=0\}$ by $(A_2,q_2)$. 
Note that even if $q$ is nondegenerate, $q_2$ may be 
degenerate.

In our equality (\ref{nikulin}), 
$A_M$ is $2$-elementary, hence $\Gamma$ is $2$-elementary and $\Gamma_T$ is
contained in $(A_T)_2$. Put $\#\Gamma=2^a$. This shows $a \le 5=l_2(A_T)$, where $l_2$ denotes 
the number of minimal generators of the $2$-Sylow subgroup of $A_T$. 

 Also it follows 
\begin{equation}\label{nikulin2}
((A_M\oplus A_T)_2|_{\Gamma^{\perp}})/\Gamma 
 \subset (A_M\oplus A_T|_{\Gamma^{\perp}})/\Gamma=A_K,
\end{equation}
since $\Gamma$ is $2$-elementary.
$(A_M\oplus A_T)_2$ has a radical of order $2$ contained in 
$(A_T)_2$. Since $\Gamma$ is a graph, this radical is not contained 
in $\Gamma$. This shows that 
$\# ((A_M\oplus A_T)_2|_{\Gamma^{\perp}})=2^{15-a}$.
Thus the order of the left-hand-side of (\ref{nikulin2}) is $2^{15-2a}$.
Since $K$ is of rank $7$, we have $15-2a\le 7$ and hence $a=4,5$.

We show that if $a=5$ then $K$ contains $(-2)$-elements and contradicts 
the assumption.
For this, first note that in this case
$\Gamma_T=(A_T)_2$ is uniquely determined 
and the embedding of $\Gamma_T$ in $A_M\simeq u(2)^{\oplus 5}$ is unique up to isomorphism by 
Witt's theorem. So we can compute $q_K$ directly and 
get $q_K\simeq u(2)^{\oplus 2}\oplus \bracket{1/4}$.
From this we see that there are inclusions 
$K\subset K' \subset \bracket{-1}^{\oplus 7}$ 
such that $K'$ is an even lattice, $[K':K]=4$ and $[\bracket{-1}^{\oplus 7}: K']=2$.
By the definition of $D_7$, $K'\simeq D_7$. 
Consider the Dynkin diagram of $D_7$ and take a subgraph isomorphic to $A_6$
with vertices $e_1,\cdots e_6$ in this order. Put $f_0=0, f_j=e_1+\cdots+e_j$, $1\le j \le 6$.
Any difference of two of these seven elements have self-intersection $(-2)$.
If $K$ has no $(-2)$-elements, then $\{f_j\}_{0\le j\le 6}$ cannot be in 
the same residue class of $K'/K$. 
Then we must have $[K':K]\ge 7$ and contradiction.

Thus we obtain $a=4$. From (\ref{nikulin2}), we see that 
$\#(A_K)_2\ge 2^7$. It follows that $K(1/2)$ is an integral 
(may be odd) lattice and $\det K(1/2)=-2$. By assumption, the minimal norm 
of the positive definite lattice $K(-1/2)$ is greater than $1$.
It follows from \cite[p400, Table 15.8]{conway-sloane} that $K(1/2)\simeq 
(\bracket{-2}^{\perp}\ \mathrm{in}\ E_8)\simeq E_7$.
\end{proof}

The following nature of the lattice $K=E_7(2)$ will be used.
\begin{lem}\label{surj}
The canonical homomorphism $\sigma: O(K)\rightarrow O(q_K)$ is surjective.
\end{lem}
\begin{proof}
The same property for the lattice $E_8(2)$ is known by \cite{barth-peters}.
We reduce the lemma to this case.
Firstly, we know the orders of the two groups.
By \cite{bourbaki}, $\# O(E_7(2))= \# O(E_7)= 2^{10}\cdot 3^4 \cdot5 \cdot 7$.
On the other hand, we can easily compute the order of $O(q_K)$ 
as $\# O(q_K)=\# O(u(2)^{\oplus 2}\oplus \bracket{1/4})= 2^{10}\cdot 3^4 \cdot5 \cdot 7$
(c.f. Lemma \ref{orbit}).
Thus it is enough to show that $\sigma$ is injective.

We take a $(-4)$-element $r$ of $E_8(2)$ and identify $K$ with $r^{\perp}$.
Obviously $g\in \ker \sigma$ can be extended to an isometry $\overline{g}$ of $E_8(2)$
by defining $\overline{g}(r)=r$. It is clear that $\overline{g}$ acts on 
the discriminant $A_{E_8(2)}$ trivially. It follows from \cite[Proposition 1.7]{barth-peters}
that $\overline{g}=\pm \id.$ Since $\overline{g}(r)=r$, $\overline{g}= \id$.
\end{proof}

Because $(A_K,q_K)\simeq u(2)^{\oplus 3}\oplus \bracket{1/4}$, the next
lemma is also used.
\begin{lem}\label{orbit}
Let $(A,q)=u(2)^m \oplus \bracket{1/4}$ be a finite quadratic form.
Then the action of $O(q)$ on $A$ decomposes $A$ into $6$ orbits.
If we denote the standard generator of one of $u(2)$ by $\{e,f\}$ and 
that of $\bracket{1/4}$ by $\{g\}$,
they are as in the following table.
\begin{center}
\begin{tabular}{lll} \hline
a representative  & length & square\\ \hline
$0$ & $1$ & $0$ \\
$2g$ & $1$ & $1$ \\
$e$ & $2^{2m}-1$ & $0$ \\
$e+f$ & $2^{2m}-1$ & $1$ \\
$g$ & $2^{2m}+2^m$ & $1/4$ \\
$e+f+g$ & $2^{2m}-2^m$ & $-3/4$ \\ \hline
\end{tabular}
\end{center}
In this table, for a representative $x$, the length is $\# (O(q)\cdot x)$ and the square 
is $q(x)\in \bb{Q}/2\bb{Z}$.
\end{lem}
The proof is given by induction on $m$ and we omit it. \\

Now we are going to describe the latter set of Proposition \ref{correspondence},
i.e., we classify the Hodge structures on $N$ induced from embeddings 
$T\subset N$ as in the proposition.
We recall Definition \ref{df}, there is an isomorphism 
\begin{equation}\label{nikulin3}
q_N \simeq (q_K \oplus q_T)|_{\Gamma^{\perp}}/\Gamma, 
\end{equation}
where $\Gamma$ is the pushout of a sign-reversing isometry of subgroups 
$\Gamma_K\subset A_K$ and $\Gamma_T\subset A_T$. 
By Proposition \ref{invariance}, $\Gamma_T$ is an invariant of the Hodge structure.

We will prove the converse. Namely, suppose we have two embeddings $T\subset N_1$
and $T\subset N_2$ whose orthogonal complements are denoted by $K_i$. 
For each embedding we have $(\Gamma_i, \Gamma_{K,i} \Gamma_{T,i})$ and the equality
(\ref{nikulin3}). What we want to show is
\[(*):\Gamma_{T,1}=\Gamma_{T,2}\Rightarrow N_1\simeq N_2 {\rm (Hodge\ isometry)}.\]
The argument goes as follows. Assume we could find an isometry $\sigma_K:
A_{K_1}\rightarrow A_{K_2}$ such that the following commutes.
\[\begin{CD}
 A_{K_1} \supset\ @. \Gamma_{K,1} @>\sim>> \Gamma_{T,1} @. \ \subset A_{T}\\
@V\sigma_K VV @V\sigma_K VV @VV\id V @VV\id V \\
 A_{K_2} \supset\ @. \Gamma_{K,2} @>\sim>> \Gamma_{T,2} @. \ \subset A_{T}
\end{CD}\]
Then by Lemma \ref{surj} we can lift $\sigma_K$ to $\sigma'_K: K_1
\stackrel{\sim}{\rightarrow} K_2$ and the pair $(\sigma'_K,\id_T)$
can be lifted to an Hodge isometry $N_1\stackrel{\sim}{\rightarrow} N_2$.
Thus it is enough to find $\sigma_K$. 

By Proposition \ref{e7}, $[N:K\oplus T]=\#\Gamma=4.$ 
Thus there are two possibilities of underlying groups of $\Gamma_K\simeq \Gamma_T$.
We consider each case separately. 

\medskip
First we consider the case $\Gamma_{T,i}\simeq \bb{Z}/4\bb{Z}$. 
The square of the generator 
$g_T\in \Gamma_{T,1}=\Gamma_{T,2}$
is independent of the 
choice and there are two possibilities, $q_T(g_T)=-1/4$ or $3/4$.
Let $(g_{K,i},g_T)\in \Gamma_i$. We have $q_{K_i}(g_{K,i})=1/4$ or $-3/4$ respectively.
By Lemma \ref{orbit}, in these cases we can find $\sigma_K$ and 
$(*)$ is proved. We find easily that there are $10$ subgroups $\Gamma_T$ 
satisfying $q_T(g_T)=-1/4$. Also there are $6$ with $q_T(g_T)=3/4$.

\medskip
Second we consider the case $\Gamma_T\simeq \bb{Z}/2\bb{Z}\oplus \bb{Z}/2\bb{Z}$.
The argument becomes slightly complicated, but the conclusion is the same.
To prove $(*)$ in this case, first we show that 
$\Gamma_{i}$ always contains a particular element. 
Here, for a clear argument, we take generators $g_i$ and $g'$ of $\bracket{1/4}\subset A_{K_i}$ and
$\bracket{-1/4}\subset A_T$ respectively. We denote an element of 
$A_{K_i}\oplus A_T$ by 
\[(x,y;z,w)\in A_{K_i}\oplus A_T \ ;\ x\in u(2)^{\oplus 3},
y\in \bracket{1/4}, z\in u(2)^{\oplus 2}, w\in \bracket{-1/4}.\]
Then the claim is that 
\[(0,2g_i;0,2g')\in \Gamma_i.\]
In fact, since $\Gamma_i$ is contained in $(A_{K_i}\oplus A_T)_2$, 
the radical element $(0,2g_i;0,0)$ of $(A_{K_i}\oplus A_T)_2$ 
is in $\Gamma^{\perp}_i$.
Hence its residue class $(0,2g_i;0,0)+\Gamma_i$ determines an element of $A_{N_i}$ by the 
isomorphism (\ref{nikulin3}). It is nonzero because $q_{K_i}(2g_i)=1.$
Since $A_{N_i}$ is nondegenerate, there exists an element $(x,y;z,w)+\Gamma_i \in A_{N_i}$
with $(0,2g_i;0,0)\cdot (x,y;z,w)=1/2.$ It follows $y=\pm g_i$.
Further since $(q_{K_i}\oplus q_T)(x,y;z,w)\in\bb{Z}$, it follows $w=\pm g'$,
i.e., there exists an element in $\Gamma_i^{\perp}$ of the form $(x,\pm g_i;z,\pm g')$.
Since the residue class of this element is of order $2$ in $A_{N_i}$, we have that 
$(0,2g_i;0,2g')\in \Gamma_i$.

Let $\Gamma_{T,1}=\Gamma_{T,2}=\bracket{2g',\alpha}$. 
Replacing $\alpha$ by $\alpha+2g'$ if necessary,  
we can assume $q_T(\alpha)=0$. Let $(\beta_i,\alpha)\in \Gamma_i$, $q_{K_i}(\beta_i)=0$. 
By Lemma \ref{orbit}, we can find $\sigma_K: A_{K_1}\stackrel{\sim}{\rightarrow}
A_{K_2}$ which takes $\beta_1$ to $\beta_2$. This $\sigma_K$ must take $2g_1$ to $2g_2$,
so we have now proved $(*)$.
There are $15$ possible $\Gamma_T$
in this case.

In summary, we have obtained the following.
\begin{prop}\label{fin}
Let $X$ be a Picard-general Jacobian Kummer surface. Then 
free involutions $\sigma_1,\sigma_2$ are conjugate 
if and only if the patching subgroups $\Gamma_{\sigma_1},\Gamma_{\sigma_2}$ coincide.
There exist (at most) $31=10+15+6$ free involutions. 
\end{prop}

The existence of $31$ free involutions is assured by concrete constructions
in the following sections.

\section{The $(16)_6$ configuration on a Jacobian Kummer surface}\label{166}
 
In this section we recall and prepare notations
concerning the divisors on Jacobian Kummer surfaces. 
The content of this section is known, 
references are \cite{keum97, kondo98, dolgachev-keum}.

{\textbf{The index set.}}
Let $C$ be a smooth projective curve of genus $2$. 
It is a double cover of $\bb{P}^1$ 
which ramifies at $6$ Weierstrass points $\{p_1,\cdots ,p_6\}\subset C$.
Here we should notice the linear equivalence 
\[p_i+p_j+p_k-p_l-p_m-p_n\sim 0 \]
for an arbitrary permutation $\{i,j,k,l,m,n\}$ of $\{1,\cdots,6\}$.
The set of theta characteristics of $C$ is by definition
\[S(C)=\{D\in \Pic (C)|2D\sim K_C\}.\]
They are divided into odd theta characteristics 
$\{[p_i]|i=1,\cdots ,6\}$ and even ones 
$\{[p_i+p_j-p_k]|i,j,k \mathrm{\ are\ distinct\ each\ other}\}$. There are $16$ 
theta characteristics.

The Jacobian variety $J(C)$ consists of divisor classes of degree $0$ on $C$.
We denote by $J(C)_2$ the set of sixteen $2$-torsion points of $J(C)$. Then 
\[J(C)_2=\{0\}\cup \{[p_i-p_j]|i\neq j\}.\]

These $16+16=32$ divisor classes naturally correspond to partitions of the set 
$\{1,\cdots 6\}$ into two subsets in the following way.
\begin{gather*}
[p_i]\in S(C) \longleftrightarrow\{i\}\cup \{i\}^c.\\
[p_i+p_j-p_k]\in S(C) \longleftrightarrow\{i,j,k\}\cup \{i,j,k\}^c.\\
[p_i-p_j]\in J(C)_2 \longleftrightarrow\{i,j\}\cup \{i,j\}^c.\\
0\in J(C)_2 \longleftrightarrow \emptyset\cup\{1,\cdots,6\},
\end{gather*}
where the complement is taken in the set $\{1,\cdots,6\}$.
We denote these partitions
by exhibiting one of the subsets, surrounded by $[\ ]$.
For example, $p_1-p_2$ corresponds to $[12]=[3456]$, $p_1+p_2-p_3$ corresponds to $[123]=[456]$,
etc. $[\emptyset]$ is denoted by $[0]$.
In this notation, we see that the symmetric difference of subsets 
$\alpha,\beta$ of $\{1,\cdots,6\}$ corresponds to addition or difference in $\Div (C)$
as follows.
\begin{gather*}
[\alpha\circleddash \beta]=[\alpha]-[\beta]\ {\rm if}\ [\alpha],[\beta] \in S(C),\\
[\alpha\circleddash \beta]=[\alpha]+[\beta]\ {\rm otherwise}.
\end{gather*}
When we use a partition $[\alpha]$ as an index, $[\ ]$ will be omitted.

\medskip
{\textbf{The $(16)_6$ configuration.}}
The sixteen theta divisors on $J(C)$ corresponding to $\beta\in S(C)$ are 
\[\Theta_{\beta}=\{[p-\beta]\in J(C)|p\in C\}.\]
The sixteen nodes $\{n_{\alpha}\in \overline{X}|\alpha\in J(C)_2\}$
on $\overline{X}=J(C)/\{\pm 1\}$ are the images of $\alpha\in J(C)_2$.
On the minimal desingularization $X$ of $\overline{X}$,
$n_{\alpha}$ is blown up to give a smooth rational curve $N_{\alpha}$ on $X$.
The tropes $\overline{T_{\beta}}\subset \overline{X}$ and $T_{\beta}\subset X$
are the strict transforms of ${\Theta}_{\beta}$.
Hence we obtain $32$ $(-2)$-curves $\{N_{\alpha},T_{\beta}\}_{\alpha,\beta}$ on $X$.
The incidence relation between these divisors 
is called the $(16)_6$ configuration. It is given explicitly by 
\begin{gather*}
(N_{\alpha},N_{\alpha '})=-2\delta_{\alpha,\alpha '},\qquad 
(T_{\beta},T_{\beta '})=-2\delta_{\beta,\beta '},\\
(N_{\alpha},T_{\beta})=1 \Leftrightarrow \alpha+\beta\in\{[1],[2],[3],[4],[5],[6]\}.
\end{gather*}

A permutation of the set $\{N_{\alpha},T_{\beta}\}_{\alpha,\beta}$
which preserves the incidence relation above is called an automorphism.
Nikulin \cite{nikulin-jac} showed that the automorphism 
group is isomorphic to $(\bb{Z}/2\bb{Z})^{5}\rtimes \mf{S}_6$,
where $(\bb{Z}/2\bb{Z})^{5}$ consists of automorphisms
induced from translations by elements of $J(C)_2\cup S(C)$
and $\mf{S}_6$ acts on the index set $\{1,\cdots, 6\}.$
We took our notations as above because this $\mf{S}_6$-action is best seen.

Translations with respect to $\alpha \in J(C)_2$ are geometrically 
realized on $J(C)$. They induce automorphisms $t_{\alpha}$ of $X$.
These are the {\em translations} in the classical terms. 
In the next section we will see that translations with respect to 
$\beta\in S(C)$ are also geometrically realized by $\sigma_{\beta}\in \Aut (X)$.
These $\sigma_{\beta}$ are the {\em switches}. 
On the other hand, in general the action of $\mf{S}_6$ cannot be 
lifted to an automorphism of $X$.

{\rmk} In \cite{keum97} and \cite{kondo98}, the notations are a little different.
To adjust notations of \cite{kondo98} to ours, first we regard
$p_0$ of \cite{kondo98} as our $p_6$.
Then the correspondence is as in below.
\begin{center}
\begin{tabular}{c|cccccc}
\cite{kondo98} & $N_0$ & $N_i$ & $N_{ij}$ & $T_0$ & $T_i$ & $T_{ij}$ \\ \hline
ours & $N_0$ & $N_{i6}$ & $N_{ij}$ & $T_6$ & $T_i$ & $T_{ij6}$ 
\end{tabular}
\end{center}
\begin{lem}[\cite{kondo98}]\label{nsstr1}
For $\beta\in S(C)$, let $\Lambda(\beta):=\{\alpha\in J(C)_2|(N_{\alpha},T_{\beta})=1\}.$
Then the divisor class of
\[H=2T_{\beta}+\sum_{\alpha\in \Lambda(\beta)}N_{\alpha}\]
is independent of $\beta$ and coincides with the pullback of the hyperplane section
by the morphism $X\rightarrow \overline{X}\subset \bb{P}^3$.
\end{lem}
\begin{lem}[\cite{kondo98}]\label{nsstr2}
Assume that $X$ is Picard-general.
\begin{enumerate}
\item $NS(X)$ is generated over $\bb{Z}$ by $\{ N_{\alpha},T_{\beta}\}_{\alpha,\beta}$.
\item $\{H,N_{\alpha}\}_{\alpha}$ is an orthogonal basis of $NS(X)_{\bb{Q}}$ over $\bb{Q}$.
\item A generator set of the discriminant group $A_{NS(X)}$ is given by 
\begin{gather*}
\msf{e}_1=(N_{26}+N_{12}+N_{36}+N_{13})/2,\msf{f}_1=(N_{16}+N_{12}+N_{46}+N_{24})/2,\\
\msf{e}_2=(N_{26}+N_{12}+N_{46}+N_{14})/2,\msf{f}_2=(N_{16}+N_{12}+N_{36}+N_{23})/2,\\
\msf{g}=H/4+(N_{0}+N_{16}+N_{26}+N_{12})/2.
\end{gather*}
\end{enumerate}
\end{lem}

\medskip
{\textbf{Special sets of nodes.}}
Lastly we mention several special sets of nodes of $\overline{X}$.
See also \cite{dolgachev-keum}.
We identify the set of nodes with $J(C)_2$ which is a $4$-dimensional 
vector space over $\bb{F}_2$. We have then the symplectic bilinear form
\[([\alpha],[\alpha'])\mapsto \#(\alpha\cap \alpha') \mod 2.\]

A $2$-dimensional subspace $G$ is called {\em{G\"{o}pel}} if it
is totally isotropic. The translations of G\"{o}pel subgroups are 
called {\em{G\"{o}pel tetrads}}. There are $60$ G\"{o}pel tetrads.
A $2$-dimensional subspace $R$ which is not G\"{o}pel is called {\em{Rosenhain}} 
and its translations {\em{Rosenhain tetrads}}. There are $80$ Rosenhain tetrads.
A {\em{Weber hexad}} is a $6$-set which can be written as the symmetric difference of 
a G\"{o}pel tetrad and a Rosenhain tetrad. 
It can be shown that any Weber hexad is of one of the following forms
\begin{equation}\label{web}
\{0,ij,jk,kl,lm,mi\}\ {\mathrm{or}}\ \{ij,jk,ki,il,jm,kn\}
\end{equation}
according to whether it contains $0$ or not.
There are $192$ Weber hexads.

In the following sections, we introduce automorphisms using these special sets.

\section{Switches}\label{switch}

Switches are one kind of automorphisms 
found by F. Klein \cite{klein}. The freeness in even cases is an 
easy consequence of the description of \cite{maria}, although it is implicit there.
Let $\beta\in S(C)$.
For a smooth point $\overline{a}\in \overline{X}$, which means
that the preimage of $\overline{a}$ in $J(C)$ is $\{a,-a\}$,
the divisors $t_a(\Theta_{\beta})$ and $t_{-a}(\Theta_{\beta})$ 
intersect at two points, which is of the form  
\[t_a(\Theta_{\beta})\cap t_{-a}(\Theta_{\beta})=\{b,-b\}.\]
The switch is defined by $\sigma_{\beta}:\overline{a}\mapsto \overline{b}$.

More precisely, these switches are defined as the composite 
of the Gauss map
\[G: \bb{P}^3 \supset \overline{X}\dashrightarrow 
\overline{X}^* \subset (\bb{P}^3)^*, \]
which maps a smooth point $\overline{a}$ to $T_{\overline{a}}\overline{X}$,
and the projective linear isomorphism 
\[F_{\beta}: \overline{X}^*\rightarrow \overline{X},\]
defined for each $\beta$. See \cite{maria}.

$\sigma_{\beta}$ is a birational involution of $\overline{X}$. Hence it induces an 
involution of $X$, which we denote by the same $\sigma_{\beta}$.
We can easily check that $\sigma_{\beta}$ interchanges $N_{\alpha}$ with
$T_{\alpha+\beta}$ for $\forall \alpha\in J(C)_2$.
\begin{prop}
For an even theta characteristic $\beta$, $\sigma_{\beta}$ is a free involution on $X$.
\end{prop}
\begin{proof}
Suppose a smooth point $\overline{a}\in \overline{X}$ is 
a fixed point of $\sigma_{\beta}$.
This is equivalent to 
$t_a(\Theta_{\beta})\cap t_{-a}(\Theta_{\beta})=\{a,-a\}$
and it is necessary that $a\in t_a(\Theta_{\beta})$, $0\in \Theta_{\beta}.$
This is untrue if $\beta$ is even.

On the other hand, the divisor $N_{\alpha}$ is disjoint from $T_{\alpha+\beta}$,
so $\sigma_{\beta}$ has no fixed points.
\end{proof}

{\rmk} (1) The proof above does not use the assumption of being Picard-general.
Thus switches for even theta characteristics are always free involutions.\\
(2) The fixed point set of a switch for an odd theta characteristic is 
a curve of genus $5$, named after Humbert.\\

Let $\sigma_{\beta}$ be a free switch. 
In the following computation, we take the case $\beta=[123]$ for simplicity.
We can obtain the result for other cases by the action of $\mf{S}_6$.
Let $K$ be the 
$(-1)$-eigenspace of the action of $\sigma_{123}$ on $NS(X)$ as in Section \ref{method}.
\begin{prop}
For Picard-general $X$, $K$ is generated over $\bb{Z}$ by
the following elements.
\begin{gather*}
f=N_{15}-T_{146},\ e_2=T_{145}-N_{16},\ e_3=N_{45}-T_{6},\ e_4=T_{123}-N_{0},\\
e_5=N_{12}-T_{3},\ e_6=T_{124}-N_{34},\ e_7=N_{24}-T_{134}, \\
e_1=-(1/2)(f+2e_2+3e_3+4e_4+3e_5+2e_6+e_7).
\end{gather*}
\end{prop}
\begin{proof}
We can check that $\{f,e_2,\cdots, e_7\}$ spans a sublattice of $K$
isomorphic to $A_7(2)$. 
We now show $e_1\in NS(X)$. Modulo $NS(X)$,
\begin{eqnarray*}
e_1 &\equiv & (f+e_3+e_5+e_7)/2\\
&\equiv & (N_{15}+N_{45}+N_{12}+N_{24})/2+(T_{146}+T_{6}+T_{3}+T_{134})/2\\
&=& (N_{15}+N_{45}+N_{12}+N_{24})/2+(\frac{H}{2}\cdot 4 -\frac{1}{2}
\sum_{\alpha\in \Lambda([146])\cup\Lambda([6])\cup\Lambda([3])\cup\Lambda([134])} N_{\alpha})/2\\
&=& (N_{15}+N_{45}+N_{12}+N_{24})/2+H-\frac{1}{4}
\sum_{\alpha\in J(C)_2-\{[15],[45],[12],[24]\}} 2N_{\alpha}\\
&\equiv & \sum_{\alpha\in J(C)_2} N_{\alpha}/2.
\end{eqnarray*}
The blow up $\widehat{J(C)}$ of $J(C)$ at points of $J(C)_2$ is 
the double cover of $X$ branched exactly over $\cup_{\alpha} N_{\alpha}$.
Thus $e_1\in NS(X)$ follows.

Then it is easy to check that $\{e_1,e_2,\cdots, e_7\}$ spans a sublattice of $K$
isomorphic to $E_7(2)$. By Proposition \ref{e7}, they coincide.
\end{proof}
\begin{prop}\label{result1}
The patching subgroup of $\sigma_{123}$ is the cyclic group generated by the element 
$[x=H/4+(N_{0}+N_{12}+N_{23}+N_{31})/2]\in A_{NS(X)}$.
\end{prop}
\begin{proof}
The facts $x\in NS(X)^*$ and $y:=-(e_1+e_5+e_7)/4+e_5/2+e_6/2\in K^*$ are easily checked.
We use Lemma \ref{gamma}. We first check $x-y\in NS(X)$. This is because 
\begin{eqnarray*}
y &=& (1/8)(f+2e_2+3e_3+4e_4+5e_5+6e_6-e_7)\\
 &=& H/4 -(N_{14}+N_{24}+N_{34}+N_{56})/2
\end{eqnarray*}
and 
\begin{eqnarray*}
x-y &=& (1/2)(N_{0}+N_{12}+N_{23}+N_{31}+N_{14}+N_{24}+N_{34}+N_{56})\\
&\equiv& T_{123}-T_{4} \equiv 0.
\end{eqnarray*}
Thus $[x]\in \Gamma_{\sigma_{123}}$. Then since $[x]$ is of order $4$ in 
$A_{NS(X)}$ and $\#\Gamma_{\sigma_{123}}=4$, $\Gamma_{\sigma_{123}}$ is generated 
by $[x]$.
\end{proof}
{\obs} In the expression of $[x]$, $\{n_{0},n_{12},n_{23},n_{31}\}$ is a 
Rosenhain subgroup defined in Section \ref{166}. The class of $-x$ can be
written as $[H/4+(N_0+N_{45}+N_{56}+N_{64})/2]$, where $\{n_{0},n_{45},n_{56},n_{64}\}$
is also a Rosenhain subgroup.
In general, for an even theta characteristic $\beta$, the $6$-set $\Lambda (\beta)$
(see Lemma \ref{nsstr1}) can be uniquely written in the form $R_1\circleddash R_2$
where $R_i$ are Rosenhain subgroups. In our case $\beta=[123]$, 
$R_1=\{n_{0},n_{12},n_{23},n_{31}\}$ and $R_2=\{n_{0},n_{45},n_{56},n_{64}\}$.
\begin{prop}\label{sec}
The patching subgroup of $\sigma_{\beta}$ for general $\beta$ is generated by 
$[H/4+(\sum_{\alpha\in R}N_{\alpha})/2]$ where $R$ is one of the 
two Rosenhain subgroups corresponding to $\beta$.
\end{prop}
\begin{proof} When $\beta=[123]$, this is Proposition \ref{result1}.
Since the action of $\mf{S}_6$ is compatible with the observation above,
the general case follows.
\end{proof}

By Proposition \ref{sec}, we can write down the generator of the 
patching subgroup of the switch $\sigma_{\beta}$ for all $\beta$.
We use the notations of 
Lemma \ref{nsstr2}.

\begin{center}
\begin{tabular}{c||c|c|c|c|c}
$\beta$     & $[123]$ & [124] & [125] & [126] & [134] \\ \hline
& $\msf{e}_1+\msf{f}_2+\msf{g}$ & $\msf{e}_2+\msf{f}_1+\msf{g}$ &
 $\msf{e}_1+\msf{f}_1+\msf{e}_2+\msf{f}_2+\msf{g}$ & $\msf{g}$ &
 $\msf{f}_1+\msf{f}_2+\msf{g}$
\end{tabular}
\smallskip
\begin{tabular}{c||c|c|c|c|c}
$\beta$     & [135] & [136] & [145] & [146] & [156]  \\ \hline
& $\msf{f}_1+\msf{g}$ & $\msf{e}_1+\msf{g}$ & $\msf{f}_2+\msf{g}$ &
 $\msf{e}_2+\msf{g}$ &$\msf{e}_1+\msf{e}_2+\msf{g}$ 
\end{tabular}
\end{center}
Since all these are distinct each other, we see that 
the ten switches are not conjugate each other in $\Aut (X)$
if $X$ is Picard-general.

\section{Hutchinson's involutions associated with G\"{o}pel tetrads}

These
automorphisms appear in \cite{hut01}. The generic freeness is found by 
J. H. Keum in \cite{keum90}. We briefly recall the construction.
Let $G$ be a G\"{o}pel tetrad. If we choose $G$ as the reference points of the
homogeneous coordinates of $\bb{P}^3$, the equation of $\overline{X}$ 
becomes
\begin{eqnarray*}
A(x^2 t^2+y^2 z^2)+B(y^2 t^2+z^2 x^2)+C(z^2 t^2+x^2 y^2)+Dxyzt\\
+E(yt+zx)(zt+xy)+G(zt+xy)(xt+yz)+H(xt+yz)(yt+zx)=0,
\end{eqnarray*}
for suitable scalars $A,\cdots,H$.
$\sigma_G$ is the Cremona involution 
\[(x,y,z,t) \mapsto (1/x,1/y,1/z,1/t).\]
For a translation $t=t_{\alpha}$, we have $\sigma_{t(G)}=t\sigma_G t$, 
so that we can restrict ourselves to the case $G$ is a G\"{o}pel subgroup.
But any G\"{o}pel subgroup is of the form $\{n_{0},n_{ij},n_{kl},n_{mn}\}$
hence up to $\mf{S}_6$ we can assume 
$G_0=\{n_{0},n_{12},n_{34},n_{56}\}$.
By \cite{keum97}, the induced action of $\sigma_{G_0}$ on $NS(X)$ is given by
\begin{gather*}
N_{\alpha}\leftrightarrow H-N_{0}-N_{12}-N_{34}-N_{56}+N_{\alpha},\quad {\textrm{for\ }}
\alpha =[0],[12],[34],[56]\\
T_1\leftrightarrow T_2,\ T_3\leftrightarrow T_4,\ T_5\leftrightarrow T_6, \\
T_{134}\leftrightarrow T_{234},\ T_{123}\leftrightarrow T_{124},\ T_{125}\leftrightarrow T_{126}.
\end{gather*}
\begin{prop}
The $(-1)$-eigenspace $K$ of $\sigma_{G_0}$ is generated over $\bb{Z}$ by the following elements.
\begin{gather*}
g=N_{0}+N_{12}+N_{34}+N_{56}-H,
e_5=T_1-T_2,\ e_1=T_3-T_4,\ e_7=T_5-T_6, \\
f=T_{134}-T_{234},\ e_3=T_{123}-T_{124},\ h=T_{125}-T_{126},\\
e_2=(1/2)(f+g+h-e_3),e_4=(1/2)(f-e_1-e_3-e_5),\\ 
e_6=(1/2)(f+h-e_5-e_7).
\end{gather*}
\end{prop}
\begin{proof}
$e_1,e_3,e_5,e_7,f,g,h\in K$ generate 
a sublattice of $K$ isomorphic to $A_1(2)^{\oplus 7}$.
It is easy to see that $e_2,e_4,e_6\in NS(X)$. For example, modulo $NS(X)$,
\begin{eqnarray*}
e_2&\equiv& (1/2)(H+N_{0}+N_{12}+N_{34}+N_{56}\\
& & + T_{123}+T_{124}+T_{125}+T_{126}+T_{134}+T_{234}) \\
&=& 2H+N_0-(1/2)\sum_{\alpha\in J(C)_2} N_{\alpha}.
\end{eqnarray*}
and as in Section \ref{switch} $e_2\in NS(X)$. $e_4,e_6$ are similar.

Then we see that
$e_1,\cdots, e_7$ span 
a lattice isomorphic to $E_7(2)$.
\end{proof}
\begin{prop}\label{result2}
The patching subgroup of $\sigma_{G_0}$ is $2$-elementary abelian and generated by
\[x=(N_0+N_{12}+N_{34}+N_{56})/2, \text{ and }y=H/2.\]
\end{prop}
\begin{proof} 
This proposition is proved in the same way as 
Proposition \ref{result1}. The corresponding element in $K^*/K$ 
is $x'=(e_1+e_3)/2, y'=(e_1+e_5+e_7)/2$ and we can check 
$x-x', y-y'\in NS(X)$. Then we use Lemma \ref{gamma}.
\end{proof}
By the $\mf{S}_6$-symmetry, we obtain the following.
\begin{prop}
For any G\"{o}pel subgroup $G$, we have 
$\Gamma_{\sigma_{G}}=\bracket{H/2,(1/2)\sum_{\alpha\in G}N_{\alpha}}.$
\end{prop}
More generally, ussing the translation relation $\sigma_{t(G)}=t\sigma_G t$,
the generator above is valid for any G\"{o}pel tetrad.\\

There are $15$ G\"{o}pel subgroups. Under the notations of 
Lemmas \ref{nsstr1} and \ref{nsstr2}, we deduce the following table.
\begin{center}
\begin{tabular}{ll} \hline
The tetrad  & Patching element corresponding to $x$ \\ \hline
$[0]+[12]+[34]+[56]$ &  $\msf{e}_1+\msf{f}_1+\msf{e}_2+\msf{f}_2$ \\
$[0]+[12]+[35]+[46]$ &  $\msf{f}_1+\msf{e}_2$ \\
$[0]+[12]+[36]+[45]$ &  $\msf{e}_1+\msf{f}_2$ \\
$[0]+[13]+[24]+[56]$ &  $\msf{e}_1+\msf{f}_1+2\msf{g}$ \\
$[0]+[13]+[25]+[46]$ &  $\msf{e}_1+\msf{f}_1+\msf{f}_2+2\msf{g}$ \\
$[0]+[13]+[26]+[45]$ &  $\msf{f}_2$ \\
$[0]+[14]+[23]+[56]$ &  $\msf{e}_2+\msf{f}_2+2\msf{g}$ \\
$[0]+[14]+[25]+[36]$ &  $\msf{f}_1+\msf{e}_2+\msf{f}_2+2\msf{g}$ \\
$[0]+[14]+[26]+[35]$ &  $\msf{f}_1$ \\
$[0]+[15]+[23]+[46]$ &  $\msf{e}_1+\msf{e}_2+\msf{f}_2+2\msf{g}$ \\
$[0]+[15]+[24]+[36]$ &  $\msf{e}_1+\msf{f}_1+\msf{e}_2+2\msf{g}$ \\
$[0]+[15]+[26]+[34]$ &  $\msf{f}_1+\msf{f}_2$ \\
$[0]+[16]+[23]+[45]$ &  $\msf{e}_1$ \\
$[0]+[16]+[24]+[35]$ &  $\msf{e}_2$ \\
$[0]+[16]+[25]+[34]$ &  $\msf{e}_1+\msf{e}_2$ \\ \hline
\end{tabular}
\end{center}
Since all these are distinct each other, we see that 
the $15$ Hutchinson involutions are not conjugate each other in $\Aut (X)$
if $X$ is Picard-general.

{\rmk} In \cite{mukai} it is shown that if $(C,G)$ is bielliptic, then
the involution $\sigma_G$ cannot be defined. 

\section{Hutchinson's involutions associated with Weber hexads}\label{reye}

These automorphisms appear in \cite{hut1,hut2}. 
The freeness is found in \cite{dolgachev-keum}.
We fix a Weber hexad $W$. 
Then the linear system $L=|\mca{O}_{\overline{X}}(2)-W|$ with the assigned base 
points at $W$ defines an another quartic model $\overline{X}_W$ of $X$ in $\bb{P}^4$,
\[\overline{X}_W:\ s_1+\cdots+s_5=0,\ \lambda_1/s_1+\cdots+\lambda_5/s_5=0,\]
where $\lambda_i$ are nonzero constants.

$\sigma_W$ is the Cremona involution 
\[\sigma_W: (s_1,\cdots, s_5)\mapsto (\lambda_1/s_1,\cdots, \lambda_5/s_5),\]
preserving $\overline{X}_W$. It is free if $X$ is Picard-general \cite{dolgachev-keum}.
For any translation $t=t_{\alpha}$, we have $\sigma_{t(W)}=t \sigma_W t$
as in the Hutchinson case. Hence we can assume that the Weber hexad 
doesn't contain $n_{0}$. Then recalling (\ref{web}) in Section \ref{166}, 
we have only one Weber hexad
$W_0=\{n_{12},n_{23},n_{31},n_{14},n_{25},n_{36}\}$ up to the action of $\mf{S}_6$.
In the following we discuss this case.

\begin{lem}[\cite{dolgachev-keum}]\label{dk}
$\sigma_{W_0}$ interchanges the following $10$ pairs of smooth rational curves.
\begin{eqnarray*}
(N_{0},T_{123}),(N_{56},T_{1}),(N_{46},T_{2}),(N_{45},T_{3}),(N_{15},T_{124}),\\
(N_{16},T_{134}),(N_{24},T_{125}),(N_{26},T_{146}),(N_{34},T_{136}),(N_{35},T_{236}).
\end{eqnarray*}
\end{lem}

\begin{prop}\label{neg}
The $(-1)$-eigenspace $K$ of $\sigma_{W_0}$ is generated over $\bb{Z}$ by the following elements.
\begin{gather*}
e_1=T_2-N_{46},\ e_2=N_{15}-T_{124},\ e_3=T_1-N_{56},\\
e_4=N_0-T_{123},\ e_5=T_3-N_{45},\ e_6=N_{34}-T_{136},\\
e_7=N_{23}-N_{56}-N_{34}-N_{24}-T_{134}-T_{124}.
\end{gather*}
\end{prop}
\begin{proof}
By computing the determinant, we can see that $10$ divisors 
\begin{equation}
N_0+T_{123},\cdots,N_{35}+T_{236}\label{inv}
\end{equation}
from Lemma \ref{dk} span over $\bb{Q}$ the invariant sublattice.
$e_1,\cdots,e_6\in K$ is easy. $e_7\in K$ follows from the fact that $e_7$ is orthogonal 
to all of the divisors in (\ref{inv}). 
Then $e_1,\cdots,e_7$ spans the lattice $E_7(2)\simeq K$.
\end{proof}
{\rmk} The action of $\sigma_{W}$ on $NS(X)$ is very complicated, but essentially
we can write down this action using the proposition above. In fact we find the 
following. 

Let $W$ be a general Weber hexad. The ``degree $1$ part'' $W_1$ of $W$ is 
the set
\[\{\beta\in S(C)| (T_{\beta},\sum_{\alpha\in W}N_{\alpha})=1\}.\]
$W_1$ consists of $6$ elements. We have a natural bijection 
$\mu :W\rightarrow W_1$ defined by $(N_{\alpha},T_{\mu (\alpha)})=1$.
On the other hand, for $\alpha\not\in W$, we have the unique decomposition
\[W=G\circleddash R,\,\, G\cap R=\{n_{\alpha}\}.\] 
Let $R^{\perp}$ be the $2$-dimensional affine subspace of $J(C)_2$ which is orthogonal
to $R$ and contains $n_{\alpha}$.
Then $R\circleddash R^{\perp}$ is a Rosenhain hexad, i.e., $R\circleddash R^{\perp}$
is of the form $\Lambda (\beta)$ for some $\beta\in S(C)$.
This defines a bijection $\mu': J(C)_2-W\rightarrow S(C)-W_1$, $\alpha\mapsto 
\mu'(\alpha)=\beta$. Using these data, the action of $\sigma_W$
is as follows.
\begin{eqnarray*}
\sigma_W (N_{\alpha})= 3H- (\sum_{\alpha\in J(C)_2}N_{\alpha})/2-(\sum_{\alpha\in W}N_{\alpha})
-T_{\mu(\alpha)}, \text{ if } \alpha\in W.\\
\sigma_W (N_{\alpha})=T_{\mu'(\alpha)} \text{ if } \alpha\not\in W.\\
\sigma_W (H)=9H -(\sum_{\alpha\in J(C)_2}N_{\alpha})-4(\sum_{\alpha\in W}N_{\alpha}).
\end{eqnarray*}

\begin{prop}\label{result3}
The patching subgroup of $\sigma_{W_0}$ is cyclic and generated by 
\[x=(3/4)H+(1/2)(N_{12}+N_{23}+N_{31}+N_{14}+N_{25}+N_{36}).\]
\end{prop}
\begin{proof}
The corresponding element in $K^*/K$ is 
\[y=\frac{1}{4}e_1+\frac{1}{2}e_2+\frac{1}{2}e_4+\frac{3}{4}e_5+\frac{1}{4}e_7,\]
and we check $x-y\in NS(X)$.
\end{proof}
By the $\mf{S}_6$-symmetry and the translation relation, we obtain 
\begin{prop}
For general $W$, the patching subgroup of $\sigma_W$ is
\[\Gamma_{\sigma_W}=\bracket{(3/4)H+ (\sum_{\alpha\in W}N_{\alpha})/2}.\]
\end{prop}

There are $12$ Weber hexads modulo translations. 
One more relation is hidden in the remark above. For $\alpha\not\in W$, 
we have the unique decomposition $W=G\circleddash R,\,\, G\cap R=\{n_{\alpha}\}$.
Let $R^{\perp}$ be the orthogonal complement of $R$ at $n_{\alpha}$ and let 
$W_{\alpha}^{\perp}$ be the Weber hexad $G\circleddash R^{\perp}$. 
Then $\sigma_W$ and $\sigma_{W_{\alpha}^{\perp}}$ are conjugate, 
related by $\sigma_{W_{\alpha}^{\perp}}= \sigma_{\mu'(\alpha)}\sigma_W\sigma_{\mu'(\alpha)}$.
Modulo this relation, 
we have $6$ Weber hexads. Under the notations of Lemmas \ref{nsstr1} and \ref{nsstr2},
their patching subgroups are as follows.
\begin{center}
\begin{tabular}{ll} \hline
Weber hexad  & patchings \\ \hline
$[12]+[23]+[31]+[14]+[25]+[36]$ &  $\msf{e}_1+\msf{f}_1+\msf{e}_2+\msf{g}$ \\
$[12]+[13]+[23]+[24]+[15]+[36]$ &  $\msf{f}_1+\msf{e}_2+\msf{f}_2+\msf{g}$ \\
$[23]+[13]+[12]+[34]+[25]+[16]$ &  $\msf{e}_2+\msf{f}_2+\msf{g}$ \\
$[24]+[23]+[34]+[14]+[25]+[36]$ &  $\msf{e}_1+\msf{f}_1+\msf{f}_2+\msf{g}$ \\
$[25]+[23]+[35]+[54]+[21]+[36]$ &  $\msf{e}_1+\msf{e}_2+\msf{f}_2+\msf{g}$ \\
$[26]+[23]+[36]+[64]+[25]+[13]$ &  $\msf{e}_1+\msf{f}_1+\msf{g}$ \\
\end{tabular}
\end{center}
Thus we see that 
there are $6$ HW involutions 
that are not conjugate each other in $\Aut (X)$
if $X$ is Picard-general.

{\rmk} (1) In a forthcoming paper we will be able to determine when $\sigma_W$ is not free.\\
(2) The $6$ conjugacy classes of HW involutions are naturally ``dual'' 
to the $6$ Weierstrass points, in the sense that the $\mf{S}_6$ action on both
factors through an outer automorphism. Details are as follows.
There are $20$ Weber hexads $W$ which don't contain $n_0$ and 
conjugate each other. 
Writing $W$ uniquely as $W=G\circleddash R$ with $G\cap R=\{n_0\}$,
we can associate with such $W$ the G\"{o}pel subgroup $G$. 
But a G\"{o}pel subgroup $G=\{n_0,n_{ij},n_{kl},n_{mn}\}$ 
is determined just by 
the ``syntheme'' $(ij)(kl)(mn)\in \mf{S}_6$. Thus we obtain $20$ synthemes from the 
conjugacy class. The fact is that there appear only $10$ synthemes, and 
the synthemes not appearing here form a ``total'', which is 
the classical description of the dual of the $6$-set $\{1,\cdots,6\}$.\\
(3) The method of this paper is applicable to the case of
Picard-general quartic Hessian surfaces 
of \cite{dolgachev-keum}. In this case we have exactly one Enriques quotient.

\quad \\
{\textbf{Proof of Theorem \ref{thm2}:}
Let $N'$ be the group generated by $16$ translations $t_{\alpha}$, $16$ switches $\sigma_{\beta}$,
$16$ projections $p_{\alpha}$, $16$ correlations $p_{\beta}$, $60$ HG involutions $\sigma_G$,
$192$ HW involutions $\sigma_W$. The theorem follows from the following lemma 
as in \cite[Lemma 7.3]{kondo98}.
\begin{lem}
Let $\phi$ be an isometry of $NS(X)$ that preserves the ample cone.
Then there exists a $g\in N'$ such that $g\phi\in \Aut (D')$.
\end{lem}
\begin{proof}
Let $w'=2H-\sum N_{\alpha}/2$ be the projection of the Weyl vector $w$.
Let $v=\phi (w')$ and 
let $g\in N'$ be an element that attains the minimum $min\{(g(v), w')|g\in N'\}$.
If $r$ is the Leech root corresponding to \cite[Lemma 4.6, Case (0),(1),(2)]{kondo98},
then as in \cite{kondo98} we have $(r', g(v))>0$.

If $r$ is the Leech root corresponding to \cite[Lemma 4.6, Case (3)]{kondo98},
then it corresponds to some Weber hexad $W$ and \cite[Remark 6.3, (1)]{kondo98}
can be rewritten as 
\[4r'=3H-2\sum_{\alpha\in W}N_{\alpha}.\]
Using Proposition \ref{neg} and its Remark, we have 
\[\sigma_W (w')=w'+8r'.\]
Thus, we have 
\begin{eqnarray*}
(g(v), w')\le (g(v), \sigma_W^{-1} (w'))= (g(v),w')+8(g(v),r'), \\
(g(v),r')>0
\end{eqnarray*}
Hence $g(v)\in D'$.
\end{proof}

{\rmk} Unfortunately, $\Aut (X)$ cannot be generated only by the subset 
\[S=\{t_{\alpha},\sigma_{\beta},\sigma_{G},\sigma_W|\alpha\in J(C)_2,\beta\in S(C),
G:\text{G\"{o}pel tetrad},W:\mathrm{Weber\ hexad}\},\] 
introduced in this paper.
It is easy to see that for any element $\varphi$ written as a product of elements 
in $S$, we have $(w',\varphi(w'))\in 4\bb{Z}$. But the projection $p_{\alpha}$
have $(w',p_{\alpha}(w'))=26.$

\medskip
{\ack} The author expresses his sincere gratitude to Professor Shigeru Mukai.
He suggested using Torelli theorem for Enriques surfaces 
in proving Proposition \ref{correspondence}, 
which was a better method of counting than that of \cite{ohashi} and 
an important step for the computation in this paper.
He explained his study of \cite{mukai} in process and led the author to his conjecture.

Financial support has been provided by the Research Fellowships of the
Japan Society for the Promotion of Science for Young Scientists.


\end{document}